\documentclass[a4paper,UKenglish,cleveref, autoref, thm-restate]{lipics-v2021}

\title{Coinductive control of inductive data types} 

\author{Paige Randall North}{Department of Mathematics and Department of Information and Computing Sciences, Utrecht University, Netherlands}{}{}{This material is based upon work supported by the Air Force Office of Scientific Research under award number FA9550-21-1-0334.}

\author{Maximilien Péroux}{Department of Mathematics, University of Pennsylvania, United States}{}{}{}

\authorrunning{P.\,R. North and M. Péroux} 

\Copyright{Paige Randall North and Maximilien Péroux} 



\nolinenumbers 

\usepackage{amssymb}
\usepackage{amsmath}
\usepackage{amsthm}
\usepackage{amsfonts}
\usepackage{amsxtra}
\usepackage[all]{xy}
\usepackage{verbatim}
\usepackage{color}
\usepackage[usenames,dvipsnames]{xcolor}
\usepackage{graphicx}
\usepackage{comment}
\usepackage{tikz-cd}
\usepackage{scalerel}
\usetikzlibrary{matrix,arrows}
\usepackage{extarrows}
\usepackage{url}
\usepackage{bbm}
\usepackage{bbold}
\usepackage{stmaryrd}
\usepackage{graphicx}

\newcommand{\C}{\mathcal C}

\newcommand{\I}{\mathbb I}
\newcommand{\Set}{{\mathsf{Set}}}
\newcommand{\N}{\mathbb N}
\newcommand{\coaN}{\mathbb N^-}
\newcommand{\eN}{{\mathbb N^\infty}}

\newcommand{\bcoalg}{{\mathsf{CoAlg}}}

\newcommand{\balg}{{\mathsf{Alg}}}
\newcommand{\ubalg}{{\underline{\mathsf{Alg}}}}

\newcommand{\id}{\mathsf{id}}

\newcommand{\ind}[1]{\llbracket {#1} \rrbracket}
\newcommand{\inthom}[3]{\underline{#1}(#2,#3)}

\newcommand{\term}{*}

\newcommand{\op}{\mathsf{op}}
\newcommand{\ev}{\mathsf{ev}}
\newcommand{\standalg}[1]{{\mathbb {#1}}}
\newcommand{\standcoalg}[1]{{\mathbb {#1}^\circ}}

\renewcommand{\lim}{\mathsf{lim}}

\renewcommand{\partial}{\textsf{\reflectbox{\textup 6}}}
\newcommand{\phom}{\mu}
\newcommand{\cofree}{\mathsf{Cof}}
\newcommand{\free}{\mathsf{Fr}}

\newcommand{\tone}{\mathtt{t}}

\usepackage[colorinlistoftodos,prependcaption,textsize=tiny]{todonotes}

\newtheorem{construction}[theorem]{Construction}

\EventEditors{Paolo Baldan and Valeria de Paiva}
\EventNoEds{2}
\EventLongTitle{10th Conference on Algebra and Coalgebra in Computer Science (CALCO 2023)}
\EventShortTitle{CALCO 2023}
\EventAcronym{CALCO}
\EventYear{2023}
\EventDate{June 19--21, 2023}
\EventLocation{Indiana University Bloomington, IN, USA}
\EventLogo{}
\SeriesVolume{270}
\ArticleNo{15}

\begin{document}

\maketitle

\begin{abstract}
We combine the theory of inductive data types with the theory of universal measurings. By doing so, we find that many categories of algebras of endofunctors are actually enriched in the corresponding category of coalgebras of the same endofunctor.
The enrichment captures all possible partial algebra homomorphisms, defined by measuring coalgebras.
Thus this enriched category carries more information than the usual category of algebras which captures only total algebra homomorphisms.
We specify new algebras besides the initial one using a generalization of the notion of initial algebra.
\end{abstract}



\section{Introduction}

In both the tradition of functional programming and categorical logic, one takes the perspective that most data types should be obtained as initial algebras of certain endofunctors (to use categorical language). For instance, the natural numbers are obtained as the initial algebra of the endofunctor $X \mapsto X + 1$, assuming that the category in question (often the category of sets) has a terminal object $1$ and a coproduct $+$. Much theory has been developed around this approach, which culminated in the notion of W-types \cite{Mar84}.

In another tradition, for $k$ a field, it has been long understood (going back at least to Wraith, according to \cite{measuring}, and Sweedler \cite{sweedler}) that the category of $k$-algebras is naturally enriched over the category of $k$-coalgebras, a fact which has admitted generalization to several other settings (e.g. \cite{measuring, vasila, Per22, MRU22}).
In this paper, we extend this theory to the setting of an endofunctor on a category -- in particular those endofunctors that are considered in the theory of W-types.

This work is thus the beginning of a development of an analogue of the theory of W-types -- not based on the notion of initial objects in a \emph{category} of algebras, but rather on generalized notions of initial objects in a \emph{coalgebra enriched category} of algebras. Our main result (\cref{thm:enriched}) states that the categories of algebras of endofunctors considered in the theory of W-types are often enriched in their corresponding categories of coalgebras.
The hom-coalgebras of our enriched category carry more information than the hom-sets in the unenriched category that is usually considered in the theory of W-types.
Because of our passage to the enriched setting, we have more precise control than in the unenriched setting, and we are able to specify more data types than just those which are captured by the theory of W-types. We do this by generalizing the notion of initial algebra, taking inspiration from the notion of weighted limits.
This general theory is presented in \cref{sec:umeas}.

But first, in \cref{section: guided examples}, we begin our paper with an enlightening example which serves as an illustration of the relevance of our enriched theory and as a motivation for the more general setting. Therein, we provide explicit calculations for the case of algebras over the endofunctor $X\mapsto X+1$ on $\Set$. In that example, we illustrate that it is appropriate to interpret the elements of the hom-coalgebras as \emph{partial homomorphisms}.

Indeed, in the classical Sweedler theory, 
the enrichment in coalgebras can also be understood as encoding a notion of partial homomorphism.
Though we do not study $k$-algebras in this paper, we conclude this introduction with details of that classical theory. A \emph{measuring} from a $k$-algebra $A$ to a $k$-algebra $B$, in the sense of Sweedler \cite{sweedler}, is a $k$-coalgebra $C$ together with a linear homomorphism $\phi\colon C\otimes_k A\rightarrow B$ that is compatible with the multiplication and identities of $A$ and $B$. A measuring from $A$ to $B$ is equivalently a $k$-coalgebra $C$ together with a $k$-linear map $\phi: C \to A \to B$ such that
\[
\phi_c(aa')=\sum_{i=1}^n \phi_{c^{(1)}_i}(a)\phi_{c^{(2)}_i}(a'), \quad \text{and} \quad \phi_c(1_A)=\varepsilon(c)1_B,
\]
for all $c \in C$ and $a,a' \in A$
where $\Delta(c)= \sum_{i=1}^n c^{(1)}_i\otimes c^{(2)}_i$ is the comultiplication $\Delta\colon C\rightarrow C\otimes_k C$ and $\varepsilon\colon C\rightarrow k$ is the counit of $C$.
Therefore the $k$-linear maps $\phi_c\colon A\rightarrow B$ can be regarded as partial algebra homomorphisms, and the elements $c\in C$ can be interpreted as measuring how far each partial homomorphism $\phi_c$ is from being a total homomorphism. For instance when $\Delta(c)=c\otimes c$, we have that $\phi_c\colon A\rightarrow B$ is a total algebra homomorphism. Now we proceed to tell an analogous story about endofunctors.



\section{Illustrative example: \texorpdfstring{$\id +1$}{id + 1}}\label{section: guided examples}

In this section, we illustrate our results in the context of one example: the endofunctor that sends $X \mapsto X + 1$ (the coproduct of $X$ and a terminal set $1$) in $\Set$, the category of sets. The initial algebra of this endofunctor is $\N$, the natural numbers, and thus this endofunctor is one of the most basic and important examples in the theory of W-types.

This section is one very long worked example of our general, categorical theory which follows in \cref{sec:umeas}.

We first review the classical story in \cref{ssec:preliminaries}, and afterwards our goal is to explain how the category of algebras is naturally enriched in the category of coalgebras of this functor and how we can use this extra structure to generalize the notion of \emph{initial algebra} to capture more algebras than just $\N$. So, in \cref{ssec:partial homomorphisms} we explore by hand a notion of partial homomorphism between algebras that will be captured more formally later in the enrichment.
Next, in \cref{ssec: Composing partial homomorphisms}, we explore the structure that this enrichment gives us. In \cref{ssec: The convolution algebra}, we introduce a computational tool and compute explicitly some of the hom-objects of our enrichment.
Finally, in \cref{ssec: Generalizing initial objects}, we use this extra structure to generalize the notion of initial object, and we describe some of the algebras that can be specified in this way.

Note that many of the proofs in this paper were relegated to the appendices, which do not appear with this, published, version. Thus, we repeatedly reference proofs in the full version, \cite{NP23}.


\subsection{Preliminaries}\label{ssec:preliminaries}

Here, we review the established theory regarding algebras and coalgebras of $\id +1$ that we will use. See, for instance, \cite[Ch.~3]{Rut19} for details.

We let $\balg$ denote the category of algebras of $\id$ + 1, and we let $\bcoalg$ denote the category of coalgebras of $\id$ + 1. Recall that an algebra is a pair $(A, \alpha)$ of a set $A$ together with a function $\alpha: A + 1 \to A$ (equivalently, a \emph{successor} endofunction $\alpha|_A : A \to A$ and a \emph{zero} $\alpha|_1 : 1 \to A$), and a coalgebra is a pair $(C, \chi)$ of a set $C$ together with a function $\chi: C \to C + 1$, i.e., a partial endofunction. The initial object of $\balg$ is $(\N, \alpha_\N)$, where $\N$ is the usual natural numbers, $\alpha_\N|_\N$ is the usual successor function  $x \mapsto x + 1$ and $\alpha_\N|_1$ picks out $0 \in \N$. The terminal object of $\bcoalg$ is $(\eN, \chi_\eN)$ where $\eN$ is the extended natural numbers $\N + \{\infty\}$, and the map $\chi_\eN : \eN \to \eN + 1$ takes $0 \in \eN$ to the element $\tone \in 1$ and all other $x \in \eN$ to $x - 1 \in \eN$.

Note that because $\N$ is initial in $\balg$, any algebra $(A, \alpha_A)$ gets a function $!_A: \N \to A$, and thus it will be useful write $n_A$ for $!_A(n)$. That is, $0_A$ is the zero of $A$, $1_A$ is the successor of $0_A$, etc. For $a \in A$, we will often also write $a + 1$ as shorthand for $\alpha_A (a)$ (especially when the algebra structure morphism, here $\alpha_A$, does not have an explicit name).

Dually, because $\eN$ is terminal in $\bcoalg$, there is a function $\ind{-} : C \to \eN$ for any coalgebra $(C, \chi_A)$, and we will say that the \emph{index} of a $c \in C$ is $\ind{c}$. Then the elements of $C$ that have index $0$ are those $c$ such that $\chi_C(c) = \tone$, those that have index $1$ are all those other $c$ such that $\chi_C^2(c)  = \tone$, etc. For $c \in C$ where $\ind{c} \neq 0 $, we will also often write $c -1$ to denote $\chi_C (c)$ (especially when the coalgebra structure morphism does not have an explicit name).

Besides $\N$, the \emph{initial algebra}, we will often consider \emph{preinitial algebras}, that is, algebras $A$ for which $!_A: \N \to A$ is epic. The nontrivial preinitial algebras are of the form
$\standalg{n} := (\{0,1,...,n\}, \alpha_\standalg{n})$ for any $n \in \N$. Here, $\alpha_\standalg{n}$ is the algebra structure that $\{0,1,...,n\}$ inherits as the quotient of $\N$ in $\Set$ that identifies all $m \geq n$ (see \cite[Example 39]{NP23} and the preceding \cite[Lemma 38]{NP23}).

Dually, besides $\eN$, we will often consider \emph{subterminal coalgebras}, that is, coalgebras $C$ for which $\ind{-}: C \to \eN$ is monic. The nontrivial ones are $\standcoalg{n}$ with underlying subset $\{0,1,...,n\}$ of $\eN$, $\coaN$ with underlying subset $\N$, and $\I$ with underlying subset $\{\infty\}$. These all inherit coalgebra structures from $\eN$ (see \cite[Example 43]{NP23} and the preceding \cite[Lemma 42]{NP23}).

\subsection{Partial homomorphisms}\label{ssec:partial homomorphisms}
Consider algebras $A$ and $B$. We are, first of all, most interested in algebra homomorphisms $f:A\rightarrow B$ (which we might call \emph{total} algebra homomorphisms to distinguish them from the notion of \emph{partial} algebra homomorphisms which we are about to introduce). 
This means that we have $(\mathsf{H1})$ $ f(0_A)=0_B$ and $(\mathsf{H2})$ $f(a+1)=f(a)+1$ for all $a\in A$. If $A$ is $\N$, we know that there is a total algebra homomorphism $\N \to B$, and we can use $(\mathsf{H1})$ and $(\mathsf{H2})$ to inductively construct this homomorphism. 

But depending on the nature of $A$ and $B$, it might happen that we can only guarantee $(\mathsf{H1})$ and $(\mathsf{H2})$ hold for some, but not all, $a \in A$, and thus an attempt to construct a total algebra homomorphism $A \to B$ inductively might fail at some point. Perhaps $A$ is a preinitial algebra $\standalg{n}$ and $B$ is $\N$. We can try to construct a total homomorphism, so we set $f(0_\standalg{n}) :=0_\N$ following $(\mathsf{H1})$, $f(1_\standalg{n}) :=1_\N$ following $(\mathsf{H2})$, etc. This works until we get to $n_\standalg{n}$. Since $n_\standalg{n}$ is the successor of both ${(n-1)}_\standalg{n}$ and ${n}_\standalg{n}$, $(\mathsf{H2})$ tells us to send $n_\standalg{n}$ both to $n_\N$ and $(n+1)_\N$. We might say that induction worked only up to the $n$th step, or that we can define a \emph{$n$-partial homomorphism}.

We formalize this idea in the following way, in which we inductively construct partial homomorphisms in an attempt to approximate a total homomorphism. In our first attempt at formalizing this idea, \cref{con: partial induction} below, we make the simplifying assumption that $A$ is preinitial -- simplifying because then there is at most one homomorphism $A \to B$. We will almost immediately drop this assumption in the more general \cref{definition: partial homomorphism}.

\begin{construction}[Partial induction]\label{con: partial induction}
We seek to inductively approximate a homomorphism $A \rightarrow B$ when $A$ is preinitial. We define a sequence of functions $f_c: A \to B$ as follows.
\begin{description}
    \item[\emph{Initial step (P1)}] Define $f_0\colon A\rightarrow B$ by $f_0(a):=0_B$ for all $a\in A$. 
    \item[\emph{Inductive step}] Define $f_{c+1} \colon A\rightarrow B$ by:
    \begin{description}
        \item[\emph{(P2)}] $f_{c+1}(0_A):=0_B$;
        \item[\emph{(P3)}] $f_{c+1}(a+1):=f_{c}(a)+1$.
    \end{description}
\end{description}
We stop when $f_{c+1}$ is not well-defined.

If we have defined $f_c$ for all $c \in \N$, then we will say that we have defined an \emph{$\infty$-partial homomorphism}.
Otherwise, if we have only defined $f_c$ for all $c \in \{0, ..., n\}$, we will say that we have defined an \emph{$n$-partial homomorphism}.
\end{construction}

Since $A$ is preinitial, it is of the form $\standalg{n}$ or $\N$, and so every element is of the form $x_A$ for $x \in \N$. Thus, $f_c(x_A)$ is $x_B$ for $x \leq c$ and otherwise $c_B$. In particular, if $A = \standalg{n}$, then $f_n = f_m$ for all $m \geq n$. Now we can see that there is an $\infty$-partial homomorphism $A \to B$ if and only if there is a total homomorphism $f: A \to B$. Indeed, if we have an $\infty$-partial homomorphism $A \to B$ consisting of an $f_c : A \to B$ for all $c \in \N$, then we can define a `diagonal' total homomorphism $f: A \to B$ by $f(x_A) := f_x (x_A)$.
Conversely, if we have a total homomorphism $f: A \to B$, there is no obstruction to the inductive steps in defining a $\infty$-partial homomorphism. Thus, we can conflate the notions of a $\infty$-partial homomorphism and a total homomorphism $A \to B$ when $A$ is preinitial.

Notice that in the term ``$n$-partial homomorphism'' in the above \cref{con: partial induction}, $n$ takes values in $\eN$, the terminal coalgebra of our endofunctor. In fact, the above construction follows the
similar pattern of the measurings of algebras over a field that we mentioned in the introduction.  So now we make the following definition in which we encode the coalgebra directly. 
This allows us to generalize \cref{con: partial induction}, dropping the hypothesis that $A$ is preinitial.

\begin{definition}
    [Measuring, cf. \cref{definition: measuring coalgebra} and \cref{prop: measuring for closed}]
    \label{definition: partial homomorphism}
    Consider algebras $A$ and $B$ and a coalgebra $C$.
    A \emph{measuring from $A$ to $B$ by $C$} is a function $f\colon C \rightarrow A \rightarrow B$ such that:
    \begin{description}
        \item[\emph{(M1)}] $f_c (0_A)=0_B$ for all $c \in C$;
        \item[\emph{(M2)}] $f_c(a+1)=0_B$ for all $\llbracket c \rrbracket=0$ and for all $a\in A$;
        \item[\emph{(M3)}] $f_c(a+1)=f_{c-1}(a)+1$ for $\llbracket c \rrbracket \geq 1$ and for all $a\in A$. 
    \end{description}

    We write $\phom_C(A,B)$ for the set of measurings from $A$ to $B$. This defines a functor $\phom: \bcoalg^\op \times \balg^\op \times \balg \to \Set$.

    For a measuring $f$ and an element $c \in C$, we call $f_c$ a \emph{$C$-partial homomorphism}.
\end{definition}

\begin{example}
    Suppose that $A$ is preinitial, so that in particular every element of $A$ is of the form $0_A$ or $a + 1$.
    
    Then there is a measuring from $A$ to $B$ by $\standcoalg{n}$ if the induction of \cref{con: partial induction} creates an $n$-partial homomorphism, and in this case the functions of the form $f_c$ constructed in \cref{con: partial induction} are the same as those specified in \cref{definition: partial homomorphism}.
    
    There is a measuring by $\coaN$ if the induction never fails, and again the functions $f_c$ from \cref{con: partial induction} and \cref{definition: partial homomorphism} coincide. Now, note that exhibiting a measuring by $\eN$ amounts to exhibiting a measuring by $\coaN$ together with a total algebra homomorphism $f_\infty$. For such an $A$, then, exhibiting a measuring by $\coaN$ is equivalent to exhibiting one by $\eN$.
\end{example}

The reader might wonder why \cref{definition: partial homomorphism} does not require $\mathsf{(M2')}$ $f_c(x) = 0_B$ for any $\ind{c} = 0$ and any $x$, but rather only requires $f_c(x) = 0_B$ when $\ind{c} = 0$ and $x$ is either the zero or a successor. In the previous example, when $A$ is preinitial, every $x \in A$ is either the zero or a successor, so there is no difference between these two requirements. In the following example, we consider an algebra $A$ where this is not the case, and illustrate why we only stipulate $\mathsf{(M2)}$ and not $\mathsf{(M2')}$.

\begin{example}
Consider the algebra $A$ with underlying set $\N + \N$, where we will notate the elements of the first copy of $\N$ as $n_A$, and the elements of the second copy as $n'$. The zero of $A$ is then $0_A$ and the successors are given by $n_A + 1 := (n+1)_A$ and $n' + 1 := (n+1)'$. Total homomorphisms $A \to \N$ are determined by the image of $0'$ in $\N$.

If we require $\mathsf{(M2')}$ instead of $\mathsf{(M2)}$ in \cref{definition: partial homomorphism}, then in a measuring by $\eN$, $f_0 (0') = 0_\N$ by $\mathsf{(M2')}$, and in general $f_n (n') = n_\N$ by $\mathsf{(M3)}$.

However, following \cref{definition: partial homomorphism} as written, in a measuring by $\eN$, $f_0 (0')$ may be anything, say $z \in \N$ and then general $f_n (n') = (z + n)_\N$.

Thus, \cref{definition: partial homomorphism} does generalize the idea of inductively approximating a total homomorphism $A \to \N$ from \cref{con: partial induction}.
\end{example}

Now notice another difference between \cref{con: partial induction} and \cref{definition: partial homomorphism}. In \cref{con: partial induction} we continue the induction as far as we can, but there is nothing of this flavor in \cref{definition: partial homomorphism}. For instance, if there is a total algebra homomorphism $A \to B$, then in the process described by \cref{con: partial induction}, we will inductively construct $f_c$ for all $c \in \N$. However, following \cref{definition: partial homomorphism}, we could say that $A \to B$ is measured by $\standalg{0}$ (which only amounts to exhibiting $f_0$), without making any claim about it being measured by other coalgebras -- it does not ask us to find any kind of maximum coalgebra $C$ that measures $A \to B$. To remedy this, we make the following definition.

\begin{definition}
    [Universal measuring, cf. \cref{definition: universal measuring}]
    \label{def: first universal measuring}
    Let $A$ and $B$ be algebras.

    We define the category of measurings from $A$ to $B$ to be the category whose objects are pairs $(C;f)$ of a coalgebra $C$ and a measuring $f: C \to A \to B$, and whose morphisms $(C;f) \to (D;g)$ are coalgebra homomorphisms $d : C \to D$ such that $f = g d$.

    The \emph{universal measuring from $A$ to $B$}, denoted $(\ubalg(A,B); u)$, is the terminal object in the category of measurings from $A$ to $B$. That is, if $(C;f)$ is a measuring from $A$ to $B$, then there is a unique morphism $!: C \to \ubalg(A,B)$ that makes the following diagram commute.
    \[
         \begin{tikzcd}
             C \ar[r,"f"] \ar[d,"!", dashed] & A \to B \\
             \ubalg(A,B) \ar{ur}[swap]{u}
         \end{tikzcd}
    \]
\end{definition}

\begin{example}\label{ex:universal measuring for preinitial}
Again, suppose that $A$ is preinitial.
In this case, the universal measuring is a subterminal coalgebra \cite[Lemma 37]{NP23}. We have shown that if the induction of \cref{con: partial induction} creates an $n$-partial homomorphism, then the maximum subterminal coalgebra that measures $A \to B$ is $\standcoalg{n}$, so this is the universal measuring. And if the induction of \cref{con: partial induction} creates an total homomorphism, then the maximum subterminal coalgebra that measures $A$ to $B$ is $\eN$ itself, so this is the universal measuring.
We will also show this fact more directly (i.e., without reference to \cite[Lemma 37]{NP23}) in \cref{ssec: The convolution algebra} below.
\end{example}

Since composing an arbitrary coalgebra homomorphism $C \to \ubalg(A,B)$ with $u$ produces a measuring $C \to A \to B$, we obtain a bijection, natural in $C,A,B$,
\[\phom_C(A,B) \cong \bcoalg(C, \ubalg(A,B))\]
showing that $\phom_-(A,B)$ is represented by $\ubalg(A,B)$.
In \cref{theorem: equalizer formula for universal measuring}, we will see that $\ubalg(A,B)$ always exists (for this and other endofunctors of interest). The coalgebra $\ubalg(A,B)$ will constitute the hom-coalgebra from $A$ to $B$ of our enriched category of algebras (\cref{thm:enriched}).

Now, note that a measuring by the coalgebra $\I$ is a total algebra homomorphism. Thus,
\[\balg(A,B) \cong \phom_\I(A,B) \cong \bcoalg(\I, \ubalg(A,B))\]
and so we find the hom-sets of the category of algebras can be easily extracted from $\ubalg(A,B)$ -- a statement that aligns with our intuition that $\balg(A,B)$ is the set of total algebra homomorphisms and $\ubalg(A,B)$ is the coalgebra of partial algebra homomorphisms.

\subsection{Composing partial homomorphisms}\label{ssec: Composing partial homomorphisms}
We will only prove that the universal measuring coalgebras form the hom-objects of our enriched category in \cref{thm:enriched}, but we can already work out how this enrichment behaves. Thus, in this section, we describe the composition and identities of this enriched category.

Given algebras $A, B,$ and $T$, we can always compose total homomorphisms $f: A\rightarrow B$ and $g: B\rightarrow T$ to form a total homomorphism $g \circ f: A\rightarrow T$. 
We wish to do the same for our partial homomorphisms. Consider coalgebras $C$ and $D$, a $C$-partial homomorphism $f_c:A\rightarrow B$, and a $D$-partial homomorphism $g_d : B\rightarrow T$. 
We can compose $g_d$ and $f_c$ as functions to obtain $g_d \circ f_c : A \to T$, and we claim that this is a $(D\times C)$-partial homomorphism.
Indeed $D\times C$ has a coalgebra structure where $\ind{(d,c)} = \min(\ind{d},\ind{c})$ and
$
(d,c)-1=
    (d-1, c-1)$ if $\ind{(d,c)}>0
$ for $(d,c)\in D\times C$.
This induces a symmetric monoidal structure on $\bcoalg$ for which $\I$ is the unit (\cref{prop:sym-mon}), and one can verify that $g_d \circ f_c$ is a $(D \times C)$-partial homomorphism.

Now we have constructed a function
\begin{align*}
    \phom_D(B,T) \times \phom_C(A,B)& \longrightarrow \phom_{D \times C}(A,T)\\
    (g,f) & \longmapsto g\circ f
 \end{align*}
 where $g\circ f: (D \times C)\rightarrow A\to T$ is defined by $(g \circ f)_{(d,c)}=g_d\circ f_c$. Thus, by the universal property of $\ubalg$, we obtain a function 
 \begin{align*}
     \bcoalg(D,\ubalg(B,T)) \times \bcoalg(C,\ubalg(A,B))& \longrightarrow \bcoalg(D \times C,\ubalg(A,T)).
 \end{align*}
 Applying this function to $(\id_{\ubalg(B,T)}, \id_{\ubalg(A,B)})$, we obtain a composition morphism $\circ: \ubalg(B,T)  \times \ubalg(A,B) \to \ubalg(A,T)$ such that for $(d,c) \in \ubalg(B,T)  \times \ubalg(A,B)$, we have $u_d \circ u_c = u_{d \circ c}$ (where $u$ is as in \cref{definition: universal measuring}).

Similarly, for any algebra $A$ we might ask if there is an identity $\id_A: \I \to \ubalg(A,A)$. We showed above that $\balg(A,A) \cong \bcoalg(\I, \ubalg(A,A))$. Thus, we take the image of $\id_A \in \balg(A,A)$ under this bijection.

We leave it as an exercise for the interested reader to show by hand that this constitutes an enrichment of $\balg$ in $(\bcoalg, \otimes, \I)$, i.e., that all the axioms for an enriched category are satisfied by this choice of composition and identities. We will instead leave this result (\cref{thm:enriched}) to the general setting. 
 
\subsection{The convolution algebra}\label{ssec: The convolution algebra}

We now give an alternative representation of $\phom_C(A,B)$ that can be directly defined and computed. In this section, we will be able to use it to compute $\ubalg(A,B)$ without appealing to \cite[Lemma 37]{NP23}.

 In \cref{definition: partial homomorphism}, we defined $\phom_C(A,B)$ to be a certain subset of $\Set(C, \Set(A,B)) \cong \Set(A, \Set(C,B))$. We now identify that subset as the subset of (total) algebra homomorphisms $A \to \Set(C,B)$ with a particular \emph{convolution} algebra structure on $\Set(C,B)$.

\begin{definition}
    [Convolution algebra, cf. \cref{def: second convolution algebra}]
    \label{def: first convolution algebra}
Given a coalgebra $(C, \chi_C)$ and an algebra $(B, \alpha_B)$, define the \emph{convolution algebra} $[C,B]$ to be the algebra whose underlying set is $\Set(C,B)$, whose zero is the constant function $C \to B$ at $0_B$, and where $f+ 1$ is defined by
\[
(f+1)(c)= \begin{cases}
    0_B & \text{if }\ind{c}=0;\\
    f(c-1)+1 &  \text{if }\ind{c}>0.
\end{cases}
\]
This defines a functor $[-,-] : \bcoalg^\op \times \balg \to \balg$.
\end{definition}

Given a coalgebra $(C, \chi_C)$ and an algebra $(B, \alpha_B)$, a function $m :C \rightarrow A\rightarrow B$ is a measuring if and only if the associated $\widetilde{m} : A \rightarrow C \rightarrow B$ (under the bijection $\widetilde{-}: \Set(C, \Set(A,B)) \to \Set(A, \Set(C,B))$) underlies a homomorphism $A \to [C , B]$ of algebras. Indeed, $\mathsf{(M1)}$ of \cref{definition: partial homomorphism} for $m$ is equivalent to $\mathsf{(H1)}$ $\widetilde m(0_A) = 0_{[C,B]}$ and criteria $\mathsf{(M2)}$ and $\mathsf{(M3)}$ for $f$ are equivalent to $\mathsf{(H2)}$ $\widetilde{m} (a+1) = \widetilde{m} (a) +1 $.


Therefore, we find the following string of bijections, natural in $C,A,B$,
\begin{align}\label{eq:univ prop of conv alg}
    \phom_C(A,B) \cong \bcoalg(C, \ubalg(A,B)) \cong \balg(A ,  [C,B])
\end{align}
so that we see that $\phom_C(-,B)$ is represented by $[C,B]$. We can even find a representation for $\phom_C(-,B)$, but we will leave this for the more general setting (\cref{thm: copower}). The interested reader is encouraged to calculate that other representation in this example.

In practice, when we want to compute $\ubalg(A,B)$, we will compute $[C,B]$ and then apply the universal property above. We do that now, computing some of the results of \cref{ex:universal measuring for preinitial} without appealing to \cite[Lemma 37]{NP23}.

\begin{example}
    \label{ex: compute partial alg hom}
We compute $\ubalg(\standalg{n},B)$ using the right-hand bijection in \cref{eq:univ prop of conv alg}.

We first observe the following for any coalgebra $Z$.
\[ \balg(\standalg{n} , Z) \cong 
\begin{cases} 
    * & \text{if } n_{Z} = (n+1)_{Z} \\ \emptyset & \text{otherwise}
\end{cases}
\] 
Since we are considering $Z:= [C,B]$, we need to understand when $n_{[C,B]} = (n+1)_{[C,B]}$. By definition, $0_{[C,A]}$ is the constant function at $0_A$. Then $1_{[C,A]}$ is the function that takes every $c \in C$ of index $0$ to $0_B$, and every other $c \in C$ to $1_B$. Inductively, we can show that $n_{[C,B]}(c) = \min(\llbracket c \rrbracket , n)_{B}$.
Thus, $n_{[C,B]} = (n+1)_{[C,B]}$ means that $\min(\llbracket c \rrbracket , n)_{B} = \min(\llbracket c \rrbracket , n + 1)_{B}$ for all $c \in C$, and this holds if and only if $\llbracket c \rrbracket \leq n$ for all $c \in C$ or $n_B = (n+1)_B$. Now we have the following.
\begin{align}\label{eq: conv alg calc}
    \bcoalg(C ,  \ubalg(\standalg{n},B)) \cong \balg(\standalg{n} , [C,B]) \cong
\begin{cases} 
    * & \text{if } \llbracket c \rrbracket \leq n \text{ for all } c \in C \\ 
    * & \text{if } n_B = (n+1)_B \\
    \emptyset & \text{otherwise} 
\end{cases}
\end{align}

In the case that $n_B = (n+1)_B$, we find that $\ubalg(\standalg{n},B)$ has the universal property of the terminal object, $\eN$.

Now suppose that $n_B \neq (n+1)_B$. Since $\bcoalg(C, \standcoalg{n}) = *$ if and only if $\llbracket c \rrbracket \leq n$ for all $c \in C$, $\ubalg(\standalg{n},B)$ has the universal property of $\standcoalg{n}$. 

Now we have calculated the following 
\[ \ubalg(\standalg{n},B) = \begin{cases}
    \eN \text{ if } n_B = (n+1)_B \\
    \standcoalg{n} \text{ otherwise}
\end{cases} \]
This aligns with our expectations, since there is a total homomorphism $\standalg{n} \to B$ if $n_B = (n+1)_B$ but there is only an $n$-partial homomorphism $\standalg{n} \to B$ otherwise.

Finally, note that taking $B := \N$, we have calculated $\ubalg(\standalg{n},\N)$, the \emph{dual} (\cref{ex: dual}) of $\standalg{n}$, to be $\standcoalg{n}$.
\end{example}

\subsection{Generalizing initial objects}\label{ssec: Generalizing initial objects}

Now we turn to the question of specifying algebras other than $\N$ via a generalization of the notion of initial algebra.

The fact that $\N$ is the initial object in $\balg$ means that the algebra structure on an algebra $A$ specifies exactly one total algebra homomorphism $\N \to A$, and this can be constructed inductively. Now we have introduced the notion of partial homomorphism which can be constructed by partial induction (\cref{con: partial induction}). Thus, we might ask if we can formalize a notion of being initial with respect to partial homomorphisms and partial induction.

Our calculations in this section so far have perhaps given us the intuition that the algebra $\standalg{n}$ represents $n$-partial homomorphisms in the way that $\N$ represents total homomorphisms. Indeed, from \cref{eq: conv alg calc}, there is a unique measuring $f: \standcoalg{n} \to \standalg{n} \to B$ for any algebra $B$. Now we try to capture and elucidate this fact by rephrasing it to say that $\standalg{n}$ is a certain kind of initial object with respect to such partial homomorphisms.

There are multiple equivalent definitions of initial object, and we choose the one that is amenable to generalization. We choose to define an initial object in a category $\C$ as an object $I \in \C$ such that there is a unique function $1 \to \C(I,X)$ for all $X \in \C$. Now we have brought to the surface a parameter, here $1$, that we can vary, inspired by the theory of weighted limits.

\begin{definition}
    [$C$-initial algebra, cf. \cref{def: second initial algebra}]
    \label{def: first initial algebra}
For a coalgebra $C$, we define a \emph{$C$-initial algebra} to be an algebra $A$ such that there is a unique coalgebra morphism $C \to \ubalg(A,X)$ for all algebras $X$.
\end{definition}

\begin{remark}
    One may wonder what would happen if for a \emph{set} $S$, we defined an $S$-initial algebra to be an algebra $A$ such that there is a unique \emph{function} $S \to \balg(A,X)$ for all $X \in \balg$. But every algebra is an $\emptyset$-initial algebra, and an $S$-initial algebra is an initial algebra for any $S \neq \emptyset$ (because functions $S \to T$ are unique only when $S = \emptyset$ or $T \cong 1$). Thus, we need to consider $\ubalg(A,X)$ and not just $\balg(A,X)$ to obtain interesting $C$-initial algebras.
\end{remark}

\begin{example}
    We have shown in \cref{ex: compute partial alg hom} that $\standalg{n}$ is an $\standcoalg{n}$-initial algebra.

    Since $\balg(A,X) \cong \bcoalg(\I, \ubalg(A,X))$, the initial algebra $\N$ is the (only) $\I$-initial algebra. In fact, since $\ubalg(\N,X) = \eN$ for all $X$ by \cite[Lemma 37]{NP23} or by a similar computation to \cref{ex: compute partial alg hom}, we find that $\N$ is a $C$-initial algebra for any subterminal coalgebra (i.e., $\emptyset, \standcoalg{n}, \coaN, \eN$).
\end{example}

Now we see that for instance, both $\standalg{n}$ and $\N$ are $\standcoalg{n}$-initial algebras. Thus, $\standcoalg{n}$-initial algebras are not determined up to isomorphism as initial algebras are. This captures the fact that for an algebra $B$, we can construct $n$-partial homomorphisms from both $\standalg{n}$ and $\N$ to $B$.

\begin{definition}
    [Terminal $C$-initial algebra, cf. \cref{def: second initial algebra}]
    \label{def: first terminal c-initial algebra}
    Consider the category whose objects are $C$-initial algebras, and whose morphisms $A \to B$ are total algebra homomorphisms $A \to B$. Then we call the terminal object of this category the \emph{terminal $C$-initial algebra}.
\end{definition}

\begin{example}
    Since the only $\I$-initial algebra is $\N$, it is also the terminal $C$-initial algebra.
\end{example}

We want to show that $\standalg{n}$ is the terminal $\standcoalg{n}$-initial algebra. However, we need another computational tool. This is in fact an alternate generalization of the notion of initial algebra.

Above, we might have observed that an initial object can be characterized as the limit of the identity functor and then, following the theory of weighted limits, considered objects $\lim^C\id_\ubalg$ with the following universal property.
\[ \balg(A,\lim^C\id_\ubalg) \cong \lim_{X \in \balg} \ \bcoalg (C, \ubalg(A,X))  \]
We can immediately calculate (\cref{prop:map from initial alg to dual}) that $\lim^C\id_\ubalg$ is $C^* := [C, \N]$,
the dual of $C$ (\cref{ex: dual}). By the bijection above, there is a unique total algebra homomorphism from each $C$-initial object to $\lim^C\id_\ubalg$. This will help us understand the possible structure that a $C$-initial object can have. But first, we must understand the structure of $C^*$.

\begin{example}
    Let $C:= \standcoalg{n}$. Then elements of $[\standcoalg{n}, \N]$ are sequences of $n+1$ natural numbers.
    
    The successor of a sequence $(a_i)_{i=0}^n$ is $(b_i)_{i=0}^n$ where $b_0 = 0$ and $b_{i+1} = a_i +1$. Notice that the successor of $(b_i)_{i=0}^n$ is $(c_i)_{i=0}^n$ where $c_0 = 0$, $c_1 = 1$ and otherwise $c_{i+2} = a_i +2$. Thus, we can inductively show that the $(n+1)$-st successor of any element of $[\standcoalg{n}, \N]$ is the sequence $(i)_{i=0}^n$, and the successor of this sequence is itself. 

    We claim that the unique morphism $!_{[\standcoalg{n}, \N]} : \N \to [\standcoalg{n}, \N]$ factors through $\standalg{n}$. We have $m_{[\standcoalg{n}, \N]} = (\min( i,m))_{i=0}^n$. Thus, the restriction of the map $!_{[\standcoalg{n}, \N]}$ to $\{0,...,n\} \subset \N$ is injective, and $n_{[\standcoalg{n}, \N]} = m_{[\standcoalg{n}, \N]}$ for all $m \geq n$.
\end{example}

\begin{example}\label{example: terminal C-initial algebra}
    Now we can show that $\standalg{n}$ is the terminal $\standcoalg{n}$-initial algebra. In this calculation, we use of the law of excluded middle for the only time in this paper.
    
    Consider an $\standcoalg{n}$-initial algebra $A$.

    First, we show that every $a \in A$ is either the basepoint or a successor. So suppose that there is an element $a \in A$ that is not a basepoint or successor, and consider an algebra $B$ with more than one element. Then for any $b \in B$ and any measure $f: \standcoalg{n} \to A \to B$, we can form a measure $\tilde f: \standcoalg{n} \to A \to B$ such that $\tilde f_n(a) = b$ and $\tilde f$ agrees with $f$ everywhere else, since \cref{definition: partial homomorphism} imposes no requirements on $\tilde f_n(a)$. Thus, there are multiple measures $\standcoalg{n} \to A \to B$, equivalently total algebra homomorphisms $\standcoalg{n} \to \ubalg(A, B)$, so we find a contradiction.

    Now, we consider the unique map $A \to [\standcoalg{n},\N]$ and claim that this factors through the injection $\standalg{n} \to [\standcoalg{n},\N]$, so that there is a unique $A \to \standalg{n}$. Since every element of $A$ is either a basepoint or a successor, every element of $A$ is either of the form $n_A$ or has infinitely many predecessors. The elements of the form $n_A$ are mapped those to of the form $n_{[\standcoalg{n},\N]}$, and the elements who have infinitely many predecessors can only be mapped to the `top element' $n_{[\standcoalg{n},\N]} = (i)_{i=0}^n$, since this is the only element which has an $m$-th predecessor for any $m \in \N$. Thus, the unique $A \to [\standcoalg{n},\N]$ indeed factors through $\standalg{n}$.
\end{example}

Thus we have shown how to specify algebras of the form $\standalg{n}$ in a way analogous to the specification of $\N$ as an initial algebra. After determining an algebra structure on a set $A$, we obtain a unique $n$-partial algebra homomorphism $\standalg{n} \to A$.

\section{General theory}\label{sec:umeas}

In this section, we now generalize the results of the previous section. So fix a symmetric monoidal category $(\C, \otimes, \I)$ and a lax symmetric monoidal endofunctor $(F, \nabla, \eta)$ (defined below in \cref{definition: lax symmetric monoidal}) on $\C$.

\subsection{Measuring coalgebras}

In this section, we define the general notion of \emph{measuring} for $F$. Note that in \cref{ssec:partial homomorphisms} above, it was convenient to define a measuring to be a certain kind of function $C \to A \to B$, but here we first define the notion of measuring without requiring the monoidal structure to be closed. That is, we define a measuring to be a certain kind of function $C \otimes A \to B$.

\begin{definition}\label{definition: lax symmetric monoidal}
 That $(F, \nabla, \eta)$ is a \emph{lax symmetric monoidal endofunctor} means that $F$ is an endofunctor on $\C$ with
 \begin{description}
     \item[\emph{(L1)}] a natural transformation $\nabla_{X,Y}\colon F(X)\otimes F(Y)\longrightarrow F(X\otimes Y)$, for all $X,Y\in \C$; and
     \item[\emph{(L2)}] a morphism $\eta: \I\rightarrow F(\I)$ in $\C$;
 \end{description}
 such that $(F, \nabla, \eta)$ is associative, unital and commutative, as described in \cite[Appendix A.2]{NP23}.
\end{definition}

\begin{example}
    In \cref{section: guided examples}, we considered the (cartesian closed) symmetric monoidal category $(\Set, \times, 1)$. For the endofunctor $\id + 1$, we define $\nabla_{X,Y}: (X + 1) \times (Y + 1) \to (X \times Y) + 1$ to take $(x,y) \mapsto (x,y)$, $(\tone, y) \mapsto \tone$, $(x,\tone) \mapsto \tone$, $(\tone,\tone) \mapsto \tone$ for $x \in X, y \in Y, \tone \in 1$. We define $\eta : 1 \to 1 + 1$ to be the inclusion into the first summand.
\end{example}

\begin{definition}
    [Measuring, cf. \cref{definition: partial homomorphism}]
    \label{definition: measuring coalgebra}
    Consider algebras $(A,\alpha)$ and $(B, \beta)$, and a coalgebra $(C, \chi)$. 
   We call a map $\phi:C\otimes A\rightarrow B$ a \emph{measuring from $A$ to $B$} if it makes the following diagram commute.
    \[
     \begin{tikzcd}[row sep=0]
       & F(C)\otimes F(A) \ar{r}{\nabla_{C,A}} & F(C\otimes A) \ar{r}{F(\phi)} & F(B) \ar{dd}{\beta}\\
      C\otimes F(A) \ar{ur}{\chi\otimes \id} \ar{dr}[swap]{\id\otimes \alpha} \\
      & C\otimes A \ar{rr}{\phi} & & B
     \end{tikzcd}
    \]
    We denote by $\mu_C(A,B)$ the set of all measurings $C\otimes A\rightarrow B$.
\end{definition}

If $\phi\colon C\otimes A \rightarrow B$ is a measuring, $a\colon (A', \alpha') \rightarrow (A, \alpha)$ and $b\colon (B, \beta) \rightarrow (B', \beta')$ are algebra homomorphisms, and $c\colon (C', \chi')\rightarrow (C, \chi)$ is a coalgebra homomorphism, then one can check that the composite
\[
\begin{tikzcd}
   C'\otimes A' \ar{r}{c\otimes a} & C\otimes A \ar{r}{\phi} & B\ar{r}{b} & B'
\end{tikzcd}
\]
is a measuring. Therefore, the assignment $C,A,B \mapsto \mu_C(A,B)$ underlies a functor
\[
\mu \colon \bcoalg^\op \times \balg^\op \times \balg \longrightarrow \Set. 
\]
We shall see that this functor is representable in each of its variables under reasonable hypotheses.

\begin{example}
    The monoidal unit $\I$ of $\C$ is a coalgebra via the lax symmetric monoidal structure $\eta:\I\rightarrow F(\I)$. Thus morphisms $A \to B$ in $\C$ are in bijection with morphisms $\I \otimes A \to B$ in $\C$, and one can check that a morphism $A \to B$ in $\C$ is an algebra homomorphism if and only if $\I \otimes A \to B$ is a measure. Thus, $\mu_\I( A, B)\cong\balg(A,B)$.
\end{example}

\begin{definition}[Universal measuring, cf. \cref{def: first universal measuring}]
    \label{definition: universal measuring}
    Let $A$ and $B$ be algebras.

    We define the category of measurings from $A$ to $B$ to be the category whose objects are pairs $(C;f)$ of a coalgebra $C$ and a measuring $f: C \otimes A \to B$, and whose morphisms $(C;f) \to (D;g)$ are coalgebra homomorphisms $d : C \to D$ such that $f = g (d \otimes A)$.

    The \emph{universal measuring from $A$ to $B$}, denoted $(\ubalg(A,B), \ev)$, is the terminal object (if it exists) in the category of measurings from $A$ to $B$. That is, if $(C;f)$ is a measuring from $A$ to $B$, then there is a unique morphism $!: C \to \ubalg(A,B)$ that makes the following diagram commute.

    \[
\begin{tikzcd}
    C\otimes A \ar{r}{f} \ar[dashed]{d}[swap]{! \otimes A} & B\\
    {\ubalg(A,B)\otimes A}\ar{ur}[swap]{\ev} 
\end{tikzcd}
\]

\end{definition}

If a universal measuring $(\ubalg(A,B), \ev)$ exists, then we obtain a representation $\ubalg(A,B)$ for $\mu_-(A,B) : \bcoalg^\op \to \Set$. That is, we have the following bijection, natural in $C,A,B$.
\[
\mu_C(A,B)\cong \bcoalg(C, \ubalg(A,B)).
\]

In the following sections, we will show that if $\C$ is closed and locally presentable and $F$ is accessible, then the universal measuring always exists.

\subsection{Local presentability, accessibility, and the measuring tensor}

We will now usually require that $\C$ be \emph{locally presentable} and $F$ is \emph{accessible} \cite[Def.~1.17~\&~2.17]{presentable}. Then $\balg$ and $\bcoalg$ are also locally presentable, the forgetful functor $\balg \to \C$ has a left adjoint $\free$, and the forgetful functor $\bcoalg \to \C$ has a right adjoint $\cofree$ \cite[Cor.~2.75~\&~Ex.~2.j]{presentable}. We will also use that these categories, as locally presentable categories, are complete and cocomplete.

\begin{example}
    $\Set$ is locally presentable and $\id + 1$ is accessible.
\end{example}

If $\C$ is locally presentable and $F$ is accessible, then for a coalgebra $(C,\chi)$, and algebras $(A,\alpha)$ and $(B,\beta)$, a map $\phi\colon C\otimes A\rightarrow B$ uniquely determines an algebra homomorphism $\phi'\colon\free(C\otimes A)\rightarrow (B,\beta)$.
Notice then that a map $\phi\colon C\otimes A\rightarrow B$ is a measuring if and only if both composites from $\free(C\otimes FA)$ to $(B, \beta)$ coincide in the following diagram. 
\[
\begin{tikzcd}
   \free(C\otimes FA) \ar[shift left, "\free(\id_C \otimes \alpha)"]{rr} \ar[shift right,"{f}"']{rr} && \free(C\otimes A) \ar{r}{\phi'} & (B, \beta)
\end{tikzcd}
\]
In the above, $f$ is obtained as adjunct under the free-forgetful adjunction of the composition
\[
\begin{tikzcd}
   C\otimes FA \ar{r}{\chi\otimes \id}& FC\otimes FA \ar{r}{\nabla_{C,A}} & F(C\otimes A) \ar{r}{F(i)} & F(\free(C\otimes A)) \ar{r}{\alpha_\free} & \free(C\otimes A),
\end{tikzcd}
\]
in which $i$ is the unit of the free-forgetful adjunction and $\alpha_\free$ is the algebra structure on the free algebra $\free(C\otimes A)$. We have now shown the following.

\begin{theorem}
    \label{thm: copower}
Suppose that $\C$ is locally presentable and $F$ is accessible.
Consider a coalgebra $C$ and an algebra $A$.
Then the coequalizer of the following diagram in $\balg$ exists, and we denote it by $C\triangleright A$ and call it the \emph{measuring tensor of $C$ and $A$.}
\[
\begin{tikzcd}
   \free(C\otimes FA) \ar[shift left]{r} \ar[shift right]{r} & \free(C\otimes A) \ar[dashed]{r}{\mathsf{coeq}} & C\triangleright A.
\end{tikzcd}
\]
Given any algebra $B$, a measuring $\phi:C\otimes A\rightarrow B$ uniquely corresponds to an algebra homomorphism $C\triangleright A\rightarrow B$. In other words, we obtain a natural identification
\[
\mu_C(A,B)\cong \balg(C\triangleright A, B).
\]
That is, the functor $\mu_C(A,-) : \balg \to \Set$ is represented by $C \triangleright A$.
\end{theorem}

In the following sections, we will also construct representing objects for $\mu_C(-,B)$ and $\mu_-(A,B)$.

\subsection{Measurings as partial homomorphisms}

Now we will often assume that the symmetric monoidal structure of $\C$ is closed. Whenever we do, we will denote the internal hom by $\inthom{\C}{-}{-}$. In this section, we provide a dual description of measurings when $\C$ is closed, generalizing \cref{definition: partial homomorphism}.

Note that since $F$ is lax monoidal, it is also \textit{lax closed}: that is, there is a map
\[
\widetilde{\nabla}_{X,Y}:F\left( \inthom{\C}{X}{Y}\right) \stackrel{}\longrightarrow \inthom{\C}{FX}{FY}
\]
natural in  $X,Y \in \C$.
Indeed, this is the adjunct under the adjunction $- \otimes FX \dashv \inthom{\C}{FX}{-}$ of the composition
\[
\begin{tikzcd}
F(\inthom{\C}{X}{Y})\otimes F(X) \ar{r}{\nabla_{\inthom{\C}{X}{Y}, X}}  & [2em] F(\inthom{\C}{X}{Y}\otimes X) \xrightarrow{F(\ev_{X})} F(Y),
\end{tikzcd}
\]
in which $\ev_{X}$ is the counit of the adjunction $- \otimes X \dashv \inthom{\C}{X}{-}$.

Given a closed monoidal structure, we can connect the notion of measuring with our notion of partial homomorphism from \cref{section: guided examples}.
\begin{proposition}
    [cf. \cref{definition: partial homomorphism}]
    \label{prop: measuring for closed}
    Suppose that $\C$ is closed.
    Given algebras $(A, \alpha)$ and $(B,\beta)$ and a coalgebra $(C,\chi)$, a map $\phi: C\otimes A\rightarrow B$ is a measuring if and only if its adjunct $\widetilde{\phi}:C\rightarrow \inthom{\C}{A}{B}$ fits in the following commutative diagram
    \[
        \begin{tikzcd}[row sep=0]
          & F(C) \ar{r}{F(\widetilde{\phi})} & F(\inthom{\C}{A}{B}) \ar{r}{\widetilde{\nabla}_{A,B}} & \inthom{\C}{FA}{FB} \ar{dd}{\beta_*}\\
         C \ar{ur}{\chi} \ar{dr}[swap]{\widetilde{\phi}} \\
         & \inthom{\C}{A}{B} \ar{rr}{\alpha^*} & & \inthom{\C}{FA}{B}
        \end{tikzcd}
       \]
    where $\alpha^*$ denotes precomposition by $\alpha$ and $\beta_*$ denotes postcomposition by $\beta$. We shall also refer to the pair $(C; \widetilde{\phi})$ as a \emph{measuring}.
\end{proposition}


\begin{example}
Note that the cartesian monoidal structure on $\Set$ is closed, and that the above recovers \cref{definition: partial homomorphism}.
\end{example}

This approach allows us to reformulate the notion of measuring as certain coalgebra homomorphisms which we now describe. If $\C$ is locally presentable and $F$ is accessible, then given a coalgebra $(C,\chi)$ and algebras $(A,\alpha)$ and $(B, \beta)$, a map $\phi: C\rightarrow \inthom{\C}{A}{B}$ in $\C$ uniquely determines a coalgebra homomorphism $\phi': (C, \chi_C) \rightarrow \cofree(\inthom{\C}{A}{B})$. A map $\phi: C\rightarrow \inthom{\C}{A}{B}$ is a measuring if and only if both composites from $(C, \chi_C)$ to $\cofree\big( \inthom{\C}{FA}{B}\big)$ in the following diagram coincide.
\[
\begin{tikzcd}
    (C, \chi_C) \ar{r}{\phi'} & \cofree\big( \inthom{\C}{A}{B}\big) \ar[shift left,"\cofree(\inthom{\C}{\alpha}{B})"]{rr} \ar[shift right,"f"']{rr} && \cofree\big( \inthom{\C}{FA}{B}\big)
\end{tikzcd}
\]
In the above, $f$ is the adjunct under the cofree-forgetful adjunction of the following composite.
\[
\begin{tikzpicture}[baseline= (a).base]
\node[scale=0.85] (a) at (1,1){
\begin{tikzcd}
   \cofree(\inthom{\C}{A}{B}) \ar{r}{\chi_\cofree} & F \big(  \cofree(\inthom{\C}{A}{B}) \big) \ar{r}{F(\varepsilon)} & F \big( \inthom{\C}{A}{B}\big) \ar{r}{\widetilde{\nabla}_{A,B}} & \inthom{\C}{F(A)}{F(B)} \ar{r}{\beta_*} & \inthom{\C}{F(A)}{B}.
\end{tikzcd}
};  
\end{tikzpicture}
\]
Here $\chi_\cofree$ is the coalgebraic structure on the cofree coalgebra, and $\varepsilon$ is the counit of the cofree-forgetful adjunction.

Now we can use this to guarantee the existence of a universal measuring.

\begin{theorem}
    [{Proof in \cite[Appendix A.3]{NP23}}]
    \label{theorem: equalizer formula for universal measuring}
Suppose that $\C$ is locally presentable and closed and that $F$ is accessible.
Given algebras $A$ and $B$, then the universal measuring coalgebra $\ubalg(A,B)$ exists and is obtained as the following equalizer diagram in $\bcoalg$
\[
\begin{tikzcd}
   \ubalg(A,B) \ar[dashed]{r}{\mathsf{eq}} & \cofree\big( \inthom{\C}{A}{B}\big) \ar[shift left]{r} \ar[shift right]{r} & \cofree\big( \inthom{\C}{F(A)}{B}\big),
\end{tikzcd}
\]
with $\widetilde{\mathsf{ev}}\colon \ubalg(A,B)\rightarrow \inthom{\C}{A}{B}$ obtained as the composition of the equalizer map $\mathsf{eq}$ together with the counit $\cofree\big( \inthom{\C}{A}{B}\big)\rightarrow \inthom{\C}{A}{B}$ of the cofree-forgetful adjunction.
\end{theorem}

\begin{corollary}
Suppose that $\C$ is locally presentable and closed and that $F$ is accessible.
Given algebras $A$ and $B$, the functor $\mu_-(A,B) : \bcoalg^\op \to \Set$ is represented by $\ubalg(A,B)$.
\end{corollary}

\subsection{Measuring via the convolution algebra}

We will now describe the last representable object for the measuring functor.

\begin{definition}
    [Convolution algebra, cf. \cref{def: first convolution algebra}]
    \label{def: second convolution algebra}
    Suppose that $\C$ is closed.
    Given a coalgebra $(C, \chi)$ and an algebra $(A, \alpha)$ in $\C$, we define an algebra structure on $\inthom{\C}{C}{A}$, called the \emph{convolution algebra}, which is denoted $[(C, \chi), (A, \alpha)]$ or simply $[C,A]$, as follows. 
The algebra structure
$F[C,A]\rightarrow [C,A]$
is the composition
\[
\begin{tikzcd}
F(\inthom{\C}{C}{A}) \ar{r}{\widetilde{\nabla}_{C,A}} & {\inthom{\C}{FC}{FA}}\ar{r}{\alpha_*{\chi}^*} & [10pt] {\inthom{\C}{C}{A}},
    \end{tikzcd}
\]
where $\alpha_*\chi^*$ denotes postcomposition by $\alpha$ and precomposition by $\chi$.
The convolution algebra construction lifts the internal hom to a functor
\[
[-,-]\colon \bcoalg^\op \times \balg\longrightarrow \balg.
\]
\end{definition}

The convolution algebra provides a representing object for $\mu_C(-,B): \balg^\op \to \Set$. Indeed, we have the following bijection natural in $C,A,B$.
\[
\mu_C(A,B)\cong \balg(A, [C, B]).
\]
In other words, a measuring $\phi\colon C\otimes A\rightarrow B$ corresponds to an algebra homomorphism $\phi':A\rightarrow [C,B]$ under the bijection $\C(C\otimes A, B)\cong \C(A, \inthom{\C}{C}{B})$. Indeed, notice that $\phi'$ is a homomorphism if and only if the following diagram, adjunct to the one appearing in \cref{definition: measuring coalgebra}, commutes.
 \[
     \begin{tikzcd}[row sep=0]
       & F([C,B]) \ar{r}{\widetilde{\nabla}_{C,B}} & {[F(C), F(B)]} \ar{dd}{\beta_*{\chi}^*}\\
      F(A) \ar{ur}{F(\phi')} \ar{dr}[swap]{\alpha} \\
      & A \ar{r}{\phi'} &  {[C,B]}
     \end{tikzcd}
    \]



\begin{remark}\label{remark: universal measuring as an adjoint}
The convolution algebra also provides an alternative characterization of the algebra $C\triangleright A$ and coalgebra $\ubalg(A,B)$.
As limits in $\balg$ and colimits in $\bcoalg$ are determined in $\C$ \cite{varieties} and the internal hom $\inthom{\C}{-}{-}:\C^\op\times \C\rightarrow \C$ preserves limits, the functor $[-,-]:\bcoalg^\op\times \balg\rightarrow \balg$ also preserves limits. 
Moreover, fixing a coalgebra $C$, the induced functor $[C, -]:\balg\rightarrow \balg$ is accessible since filtered colimits in $\balg$ are computed in $\C$ (see \cite[5.6]{varieties}). 
Therefore, by the adjoint functor theorem \cite[1.66]{presentable}, the functor $[C,-]$ is a right adjoint. Its left adjoint is precisely $C\triangleright -:\balg\rightarrow \balg$. Indeed, for any algebras $A$ and $B$, we obtain the following bijection, natural in $C,A,B$.
\[
\balg \big( C\triangleright A, B\big) \cong \balg \big(A, [C,B]\big).
\]
Notice we can also determine the universal measuring by using the adjoint functors.
Fixing now an algebra $B$, the opposite functor $[-,B]^\op:\bcoalg\rightarrow \balg^\op$ preserves colimits, where the domain is locally presentable and the codomain is essentially locally small. By the adjoint functor theorem \cite[1.66]{presentable} and \cite[5.5.2.10]{HTT}, this functor is a left adjoint. 
Its right adjoint is precisely the functor $\ubalg(-, B):\balg^\op\rightarrow \bcoalg$. Indeed, for any algebra $A$ and $B$ and any coalgebra $C$, we have the following bijection, natural in $C,A,B$.
\[
\bcoalg\big( C, \ubalg(A,B) \big) \cong \balg \big( A, [C,B]\big).
\]
\end{remark}

Combining the identifications, we see that the measuring functor is representable in each factor:
\[
\mu_C(A,B) \cong \bcoalg\big( C, \ubalg(A,B) \big)\cong \balg \big( A, [C,B]\big) \cong \balg \big( C\triangleright A, B\big).
\]
In other words, for any algebra $A$ and $B$ and any coalgebra $C$, the following data are equivalent.
\vspace{0.2em}
\begin{center}
    \begin{tabular}{|c|c|c|c|c|}
    \hline  $C\otimes A\rightarrow B$ &  $C\rightarrow \inthom{\C}{A}{B}$ & $C\rightarrow \ubalg(A,B)$ & $A\rightarrow [C,B]$ &  $C\triangleright A\rightarrow B$ \\
      measuring \vspace{-0.08em} & measuring \vspace{-0.08em}& coalgebra  \vspace{-0.08em}& algebra\vspace{-0.08em} & algebra \vspace{-0.08em} \\ 
        &  & homomorphism & homomorphism & homomorphism \\\hline 
    \end{tabular}
\end{center}
\vspace{0.1em}

\begin{definition}
    \label{ex: dual}
    Assuming that $\C$ is locally presentable and $F$ is accessible, $\balg$ has an initial object which we denote by $N$.

    Let $(-)^*:\bcoalg^\op\rightarrow \balg$ denote the functor $[-,N]$, and call $C^*$ the \emph{dual algebra of $C$} for any coalgebra $C$.

    Let $(-)^\circ:\balg^\op\rightarrow \bcoalg$ denote the functor $\ubalg(-,N)$, and call $A^\circ$ the \emph{dual coalgebra of $A$} for any algebra $A$.
    
    These functors form a dual adjunction since we have the following bijection, natural in $C,A$:
    \[
    \balg(A, C^*)\cong \bcoalg(C, A^\circ).
    \]
\end{definition}

\subsection{Measuring as an enrichment}

We now come to the main punchline of the general theory presented in this paper: that $\ubalg(-,-)$ gives the category of algebras an enrichment in coalgebras. First, we describe how to compose measurings.

\begin{proposition}\label{prop:sym-mon}
    The category $\bcoalg$ has a symmetric monoidal structure for which the forgetful functor $\bcoalg\rightarrow \C$ is strong symmetric monoidal.
\end{proposition}

\begin{proof}
    Suppose $(C, \chi_C)$ and $(D, \chi_D)$ are coalgebras. Then $C\otimes D$ has the following coalgebra structure.
    \[
\begin{tikzcd}
{C \otimes D} \ar{r}{\chi_C \otimes \chi_D} & [1em] {F(C) \otimes F(D)} \ar{r}{\nabla_{C,D}} & {F(C\otimes D)}
\end{tikzcd}
\]
The morphism $\eta:\I\rightarrow F(\I)$ provides the coalgebraic structure on $\I$.
One can verify that $(\bcoalg, \otimes, (\I, \eta))$ is a symmetric monoidal category (see details in \cite[Appendix A.4]{NP23}).
\end{proof}

Now we can prove our main theorem.

\begin{theorem}
    [{Proof in \cite[Appendix A.5]{NP23}}]
    \label{thm:enriched}
Suppose that $\C$ is locally presentable and closed and that $F$ is accessible.
Then the category $\balg$ is enriched, tensored, and powered over the symmetric monoidal category $\bcoalg$ respectively via
\[
\balg^\op \times \balg\xrightarrow{\ubalg(-,-)} \bcoalg, \quad \bcoalg\times \balg \xrightarrow{-\triangleright-} \balg, \quad \bcoalg^\op\times \balg\xrightarrow{[-,-]} \balg.
\]
\end{theorem}

\begin{example}
    [{Details in \cite[Appendix A.7]{NP23}}]
    \label{example: further list of examples}
    Suppose that $\C$ is locally presentable and closed.
    The following endofunctors on $\C$ are accessible and lax symmetric monoidal.
    \begin{description}
        \item[$(\id)$] The identity endofunctor $\id_\C$.
        \item[$(A)$] The constant endofunctor that sends each object to a fixed commutative monoid $A$ in $\C$.
        \item[$(GF)$] The composition $GF$ of accessible, lax symmetric monoidal endofunctors $F$ and $G$.
        \item[$(F \otimes G)$] The pointwise tensor product $F \otimes G$ of accessible, lax symmetric monoidal endofunctors $F$ and $G$, assuming $\C$ is closed.
        \item[$(F + G)$] The pointwise coproduct $F + G$ of an accessible, lax symmetric monoidal endofunctor $F$ and an accessible endofunctor $G$ equipped with natural transformations $GX \otimes GY \to G(X \otimes Y)$, $\lambda: FX \otimes GY \to G(X \otimes Y)$, $\rho: GX \otimes FY \to G(X \otimes Y)$ satisfying the axioms described in \cite[Appendix A.7]{NP23}, assuming $\C$ is closed.
    \item[$(\id^A)$] The exponential $\id^A$ for any object $A$ of $\C$, assuming the monoidal product on $\C$ is cartesian closed.
    \item[($W$-types)] A polynomial endofunctor associated to a morphism $f: X \to Y$ in $\Set$, given a commutative monoid structure on $Y$ and an oplax symmetric monoidal structure on the preimage functor $f^{-1}: C \to \Set$.
    \item[(d.e.s.)] A discrete equational system, assuming that the monoidal structure on $\C$ is cocarte\-sian and that $\C$ has binary products that preserve filtered colimits.
\end{description}
\end{example}

On some occasions, the category of coalgebras of $F$ can be interesting while its category of algebras is less so. 
For instance, given an alphabet $\Sigma$, coalgebras over the endofunctor $F(X)=2\times X^\Sigma$ in $\Set$ are automata but the initial algebra remains $\emptyset$. To remedy this, we can extend our main result into the following theorem. 

\begin{theorem}[{Proof in \cite[Appendix A.6]{NP23}}]\label{theorem: mixed enriched GF over F}
    Suppose that $\C$ is locally presentable and closed and that $F$ is also accessible.
Let $G\colon\C\rightarrow \C$ be a $\C$-enriched functor that is accessible. 
Then $\balg_{GF}$ is enriched, tensored and powered over $\bcoalg_F$.
\end{theorem}

\begin{example}
If $F(X)=2\times X^\Sigma$, we could consider $G=\id+1$, and thus $\balg_{GF}$ has $\N$ as an initial object and remains enriched in automata.
\end{example}

The enrichment of algebras in coalgebras specify a pairing of coalgebras 
\[
 \ubalg(B,T)\otimes \ubalg(A,B) \longrightarrow \ubalg(A,T),
\]
regarded as an enriched composition, for any algebras $A$, $B$ and $T$.
In more details, the above coalgebra homomorphism is induced by the measuring of
\[
\begin{tikzcd}
  \big(  \ubalg(B,T)\otimes \ubalg(A,B)\big) \otimes A \ar{r}{\id\otimes \ev_{A,B}} & [2em]\ubalg(B,T)\otimes B \ar{r}{\ev_{B,T}} & T.
\end{tikzcd}
\]
In other words, the enrichment is recording precisely that we can compose a measuring $C\otimes A\rightarrow B$ with $D\otimes B\rightarrow T$ to obtain a measuring $(D\otimes C)\otimes A\rightarrow T$.
In particular, our above discussion shows that $\ubalg(A,A)$ is always a monoid object in the symmetric monoidal category $(\bcoalg, \otimes, \I)$.


\subsection{General \texorpdfstring{$C$}{C}-initial objects}

Now we generalize \cref{ssec: Generalizing initial objects}. We can use the extra structure in the enriched category of algebras to specify more algebras than we could in the unenriched category of algebras.

\begin{definition}
    [$C$-initial algebra, cf. \cref{def: first initial algebra} and \cref{def: first terminal c-initial algebra}]
    \label{def: second initial algebra}
    Suppose that $\C$ is locally presentable and closed and that $F$ is accessible.

Given a coalgebra $C$,
    we say an algebra $A$ is a \emph{$C$-initial algebra} if there exists a unique map $C\rightarrow \ubalg(C,X)$, for all algebras $X$.

The \emph{terminal} $C$-initial algebra is the terminal object, if it exists, in the subcategory of $\balg$ spanned by the $C$-initial algebras.
\end{definition}

We end with a result that helped us calculate some terminal $C$-initial algebras in \cref{ssec: Generalizing initial objects}.

\begin{proposition}
    [{Proof in \cite[Appendix A.8]{NP23}}]
    \label{prop:map from initial alg to dual}
    Suppose that $\C$ is locally presentable and closed and that $F$ is accessible.
    There is a unique map from any $C$-initial algebra to $C^*$.
\end{proposition}

\section{Conclusions \& Vista}

In this paper, we have shown that given a closed symmetric monoidal category $\C$ and an accessible lax symmetric monoidal endofunctor $F$ on $\C$, the category of algebras of $F$ is enriched, tensored, and cotensored in the category of coalgebras of $F$. The algebras of such a functor are of central importance in theoretical computer science, and we hope that identifying such extra structure can shed light on these studies. Indeed, we have demonstrated one use case: we can now specify $C$-initial algebras in an analogous way to initial algebras. We identified a large class of examples of endofunctors that are encompassed by our theory. Thus, we have established the beginning of an \emph{enriched} analogue of the theory of $W$-types. We have also worked out concretely the results for the endofunctor $\id +1 $ on $\Set$, which suggested a meaningful interpretation of the enrichment as partial algebra homomorphisms.  

In future work, we will present similar meaningful  interpretations for other endofunctors of our theory. 
Our future plans involve incorporating features, such as $C$-initial algebras, of this new enriched theory into concrete programming languages like Haskell or Agda.

We also seek to extend the results of \cref{example: terminal C-initial algebra} into more general settings and provide conditions for the existence of the terminal $C$-initial algebras. We will also develop more robust theory from \cref{theorem: mixed enriched GF over F}.
Our partial algebra homomorphisms remain total functions: it would be interesting to develop a theory that encodes maps that are partial both as a function and as algebra homomorphisms.
Lastly in \cref{example: further list of examples}, when we consider the constant functor at an object $A$, we must choose a commutative monoid structure on $A$. What if we had two different monoidal structures on $A$? There are other such choices that are needed in \cref{example: further list of examples}, for instance in our motivating example of W-types. We seek to understand how these choices interact with one another.



\bibliography{biblio}

\begin{thebibliography}{10}

\bibitem{varieties}
Ji\v{r}\'{\i} Ad\'{a}mek and Hans-E. Porst.
\newblock On varieties and covarieties in a category.
\newblock {\em Mathematical Structures in Computer Science}, 13(2):201–232,
  2003.
\newblock \href {https://doi.org/10.1017/S0960129502003882}
  {\path{doi:10.1017/S0960129502003882}}.

\bibitem{presentable}
Ji\v{r}\'{\i} Ad\'{a}mek and Ji\v{r}\'{\i} Rosick\'{y}.
\newblock {\em Locally presentable and accessible categories}, volume 189 of
  {\em London Mathematical Society Lecture Note Series}.
\newblock Cambridge University Press, Cambridge, 1994.
\newblock \href {https://doi.org/10.1017/CBO9780511600579}
  {\path{doi:10.1017/CBO9780511600579}}.

\bibitem{measuring}
Martin Hyland, Ignacio L\'{o}pez~Franco, and Christina Vasilakopoulou.
\newblock Hopf measuring comonoids and enrichment.
\newblock {\em Proc. Lond. Math. Soc. (3)}, 115(5):1118--1148, 2017.
\newblock \href {https://doi.org/10.1112/plms.12064}
  {\path{doi:10.1112/plms.12064}}.

\bibitem{HTT}
Jacob Lurie.
\newblock {\em Higher topos theory}, volume 170 of {\em Annals of Mathematics
  Studies}.
\newblock Princeton University Press, Princeton, NJ, 2009.
\newblock \href {https://doi.org/10.1515/9781400830558}
  {\path{doi:10.1515/9781400830558}}.

\bibitem{Mar84}
Per Martin-L\"of.
\newblock Intuitionistic type theory.
\newblock {\em Naples: Bibliopolis}, 1984.

\bibitem{MRU22}
Dylan McDermott, Exequiel Rivas, and Tarmo Uustalu.
\newblock Sweedler theory of monads.
\newblock In Patricia Bouyer and Lutz Schr{\"o}der, editors, {\em Foundations
  of Software Science and Computation Structures}, pages 428--448, Cham, 2022.
  Springer International Publishing.
\newblock \href {https://doi.org/10.1007/978-3-030-99253-8\_22}
  {\path{doi:10.1007/978-3-030-99253-8\_22}}.

\bibitem{NP23}
Paige~Randall North and Maximilien Péroux.
\newblock Coinductive control of inductive data types, 2023.
\newblock \href {http://arxiv.org/abs/2303.16793} {\path{arXiv:2303.16793}}.

\bibitem{Per22}
Maximilien Péroux.
\newblock The coalgebraic enrichment of algebras in higher categories.
\newblock {\em Journal of Pure and Applied Algebra}, 226(3):106849, 2022.
\newblock \href {https://doi.org/https://doi.org/10.1016/j.jpaa.2021.106849}
  {\path{doi:https://doi.org/10.1016/j.jpaa.2021.106849}}.

\bibitem{Rut19}
Jan Rutten.
\newblock {\em The Method of Coalgebra: exercises in coinduction}.
\newblock C, 2019.
\newblock URL: \url{https://ir.cwi.nl/pub/28550/}.

\bibitem{sweedler}
Moss~E. Sweedler.
\newblock {\em Hopf algebras}.
\newblock Mathematics Lecture Note Series. W. A. Benjamin, Inc., New York,
  1969.

\bibitem{vasila}
Christina Vasilakopoulou.
\newblock Enriched duality in double categories: {$\mathcal{V}$}-categories and
  {$\mathcal{V}$}-cocategories.
\newblock {\em J. Pure Appl. Algebra}, 223(7):2889--2947, 2019.
\newblock \href {https://doi.org/10.1016/j.jpaa.2018.10.003}
  {\path{doi:10.1016/j.jpaa.2018.10.003}}.

\bibitem{yetter}
D.N. Yetter.
\newblock Abelian categories of modules over a (lax) monoidal functor.
\newblock {\em Advances in Mathematics}, 174(2):266--309, 2003.
\newblock \href {https://doi.org/https://doi.org/10.1016/S0001-8708(02)00041-5}
  {\path{doi:https://doi.org/10.1016/S0001-8708(02)00041-5}}.

\end{thebibliography}

\appendix

\section{Appendix}

\subsection{Details for \texorpdfstring{\cref{section: guided examples}}{Section \ref{section: guided examples}}}

\begin{lemma}\label{lem:subobject}
        Consider two algebras $A,B$, and suppose that the canonical map $\N \to A$ is an epimorphism. Then $\ubalg(A,B)$ is a subterminal coalgebra, and it is the maximum subterminal coalgebra $C$ such that $\balg(A, [C,B])$ is nonempty.
    \end{lemma}
    
    \begin{proof}
        We first claim that $\balg(A, [C,B])$ is empty or a singleton for any coalgebra $C$.
        Since $!_A: \N \to A$ is an epimorphism, the function
        \[\balg(A, [C,B]) \xrightarrow{!^*_A}  \balg(\N, [C,B]) \]
        is an injection, and since $\N$ is initial, $\balg(\N, [C,B]) \cong *$.
        Thus, $\balg(A, [C,B])$ is a singleton, in the case where the only map $\N \to [C,B]$ factors through $A$, and otherwise is empty.

        We now claim that the canonical map $\ind{-}: \ubalg(A,B) \to \eN$ is a monomorphism. Suppose there are maps $a,b: C \to \ubalg(A,B)$ from a coalgebra $C$ such that $\ind{a} = \ind{b}$. Recall that the coalgebra $\ubalg(A,B)$ is defined by the following universal property for any coalgebra $C$
        \[ \bcoalg(C,\ubalg(A,B)) \cong \balg(A, [C,B]).\]
        Thus, $\bcoalg(C,\ubalg(A,B))$ is a singleton, so $a = b$, and $\ind{-}$ is monic.
    \end{proof}

    \begin{lemma}\label{lem: quotient algebras}
        Consider a category $\C$, an endofunctor $F: \C \to \C$, and an algebra $(A,\alpha)$ of $F$.
        \begin{enumerate}
            \item Given a quotient object $Q$ of $A$ for which $\alpha$ restricts to a morphism $\alpha_Q : FQ \to F$, the pair $(Q, \alpha_Q)$ is a quotient object of $(A, \alpha)$.
            \item If the forgetful functor $\mathsf{Alg} \to \C$ preserves epimorphisms, then for every quotient object $(Q, \alpha_Q)$ of $(A, \alpha)$, $Q$ is a quotient object of $A$.
        \end{enumerate}
        Thus, if the forgetful functor preserves epimorphisms, it underlies an isomorphism between the poset of quotient algebras of $(A,\alpha_A)$ and the poset of quotient objects $Q$ of $A$ for which $\alpha$ restricts to a morphism $\alpha_Q : FQ \to Q$.
    \end{lemma}

    \begin{proof}
        Consider a a quotient object $Q$ of $A$ for which $\alpha$ restricts to a morphism $\alpha_Q : FQ \to F$. Then $(Q, \alpha_Q)$ is an algebra of $F$. The quotient map $A \to Q$ underlies a total algebra homomorphism $q: (A, \alpha) \to (Q, \alpha_Q)$. Consider another algebra $(B, \beta)$ and two maps $f,g: (Q, \alpha_Q) \to (B, \beta)$ such that $fq = gq$. Then we find that the underlying maps of $f$ and $g$ in $\C$ coincide, and thus $f = g$.

        Suppose that the forgetful functor $\mathsf{Alg} \to \C$ preserves pushouts. Then it preserves epimorphisms
    \end{proof}

    \begin{example}\label{ex: quotient algebras for id + 1}
        Consider the endofunctor $\id + 1$ on $\Set$. This preserves surjections, so there is an isomorphism between the poset of quotient algebras of an algebra $(A,\alpha)$ and the poset of quotient sets $Q$ of $A$ for which $\alpha$ restricts to a morphism $\alpha_Q : FQ \to Q$.
    \end{example}

    \begin{definition}
        Consider a category $\C$ with a terminal object $*$, an endofunctor $F: \C \to \C$, and suppose that the category of coalgebras of $F$ has a terminal coalgebra $\eN$.

        Call a point $* \to \eN$ an \emph{index}. For any index $n$, let $\bcoalg^n$ denote the category whose objects are pairs $((X, \chi_X), x)$ where $(X, \chi_X)$ is a coalgebra and $x \in X$ such that $\ind{x} = n$.
        
        Say that a coalgebra is \emph{freely generated by a point of index $n$} if it is an initial object in $\bcoalg^n$. Say that a coalgebra is \emph{generated by a point of index $n$} if it is the vertex of a cone on the identity endofunctor on $\bcoalg^n$.
    \end{definition}

    \begin{definition}
        Say that $F : \C \to \C $ is \emph{well-equipped} if $\C$ is well-pointed and for every index $n$, there is a coalgebra generated by a point of index $n$.
    \end{definition}

    \begin{lemma}\label{lem:char-sub}
        Consider a category $\C$, an endofunctor $F$, and a coalgebra $(C, \chi)$.
        \begin{enumerate}
            \item Given a subobject $S$ of $C$ for which $\chi$ restricts to a morphism $\chi_S: S \to FS$, the pair $(S, \chi_S)$ is a subobject of $(C, \chi)$.
            \item If $F$ is well-equipped, for every subcoalgebra $(S, \chi_S)$ of $c$, we have that $S$ is a subobject of $C$.
        \end{enumerate}
        Thus, if $F$ is well-equipped, there is an isomorphism between the poset of subobjects of $(C, \chi_C)$ and the poset of subobjects of $C$ for which $\chi_C$ restricts to a morphism $\chi_S: FS \to S$.
    \end{lemma}
    \begin{proof}

    Consider a subobject $S$ of $C$ for which $\chi_C$ restricts to a morphism $\chi_S: S \to FS$. Then $(S, \chi_S)$ is a coalgebra, and there is a total algebra homomorphism $s: (S, \chi_C) \to (C, \chi)$. Consider another coalgebra $(D, \chi_D)$ and two morphisms $f,g : (C, \chi) \to (D, \chi_D) $ such that $sf = sg$. Then the underlying morphisms of $f$ and $g$ coincide, so $f = g$.

     Now suppose that $F$ is well-equipped. Consider a subcoalgebra $\iota: (S, \chi_S) \rightarrowtail (C, \chi)$, and two points $x, y: * \to S$ in $\C$ such that $\iota x = \iota y$. Thus $\ind{x} = \ind{\iota x} = \ind{\iota y} =  \ind{y}$, and let $((I, \chi_I), \top_n)$ denote a coalgebra generated by a point of this index. Then we get two maps $\overline x, \overline y: (I, \chi_I) \to (S, \chi_S)$ and a map $\overline{\iota x} = \overline{\iota y} : (I, \chi_I) \to (C, \chi)$ such that $\iota \overline x = \overline{\iota x} = \overline{\iota y} = \iota \overline y$. Since $\iota$ is a monomorphism, $\overline x = \overline y$, but then $x = \overline x \top_n = \overline y \top_n = y: * \to \C$.
    \end{proof}

    \begin{example}\label{ex: subcoalgebras for id + 1}
        Consider the endofunctor $\id + 1$ on $\Set$. The indices are the elements of $\eN$. The coalgebra $\standcoalg{n}$ is freely generated by a point of index $n$ for finite $n$, and $\I$ is freely generated by a point of index $\infty$.

        Thus there is an isomorphism between the poset of subcoalgebras of any coalgebra $(C, \chi)$ and the poset of subsets $S$ of $C$ for which $\chi$ restricts to a morphism $\chi_S: FS \to S$.
    \end{example}
\subsection{Axioms for \texorpdfstring{\cref{definition: lax symmetric monoidal}}{Definition \ref{definition: lax symmetric monoidal}}}\label{subsection appendix: diagrams for lax symmetric monoidal functors}
The triplet $(F, \nabla, \eta)$ must satisfy the following commutative diagrams.
\begin{description}
    \item[Associative] For all objects $X,Y, Z\in \C$, we have
\[
\begin{tikzcd}[row sep= large]
 F(X)\otimes F(Y)\otimes F(Z) \ar{r}{\nabla_{X,Y}\otimes \id}\ar{d}[swap]{\id\otimes \nabla_{Y,Z}} & [2.5em] F(X\otimes Y) \otimes F(Z) \ar{d}{\nabla_{X\otimes Y, Z}}\\
 F(X)\otimes F(Y\otimes Z) \ar{r}{\nabla_{X, Y\otimes Z}} & F(X\otimes Y \otimes Z).
   \end{tikzcd}
   \]
   
   \item[Unital] For all objects $X\in \C$
\[
\begin{tikzcd}
F(X)\ar[equals]{ddrr}\ar{r}{\cong} & [-1em] \I\otimes F(X)  \ar{r}{\eta\otimes \id} & F(\I)\otimes F(X)\ar{d}{\nabla_{\I, X}} & F(X) \ar{r}{\cong} \ar[equals]{ddrr} & [-1em]F(X)\otimes \I \ar{r}{\id\otimes \eta} & F(X)\otimes F(\I) \ar{d}{\nabla_{X, \I}}\\
 & & F(\I\otimes X) \ar{d}{\cong} & & & F(X\otimes \I)\ar{d}{\cong}\\
 [-1em]  & & F(X) & & & F(X).
\end{tikzcd}
\]
\item[Commutative] For all objects $X,Y\in \C$
\[
\begin{tikzcd}
   F(X) \otimes F(Y) \ar{r}{\nabla_{X,Y}} \ar{d}{\cong}[swap]{\tau_{F(X),F(Y)}}& F(X\otimes Y)\ar{d}{F(\tau_{X,Y})} [swap]{\cong}\\
   F(Y) \otimes F(X) \ar{r}{\nabla_{Y,X}} & F(Y\otimes X),
\end{tikzcd}
\]
where $\tau_{X,Y}:X\otimes Y\rightarrow Y\otimes X$ is the symmetry natural isomorphism in $\C$.
\end{description}

\subsection{Proof of \texorpdfstring{\cref{theorem: equalizer formula for universal measuring}}{Theorem \ref{theorem: equalizer formula for universal measuring}}}\label{subsection appendix: equalizer formula for universal measuring}

\begin{proof}
By the cofree-forgetful adjunction, the coalgebra homomorphism $\mathsf{eq}:\ubalg(A,B) \to \cofree\big( \inthom{\C}{A}{B}\big)$ determines a map $\widetilde{\ev}:\ubalg(A,B)\rightarrow \inthom{\C}{A}{B}$ which must be a measuring.
We verify it is the universal one.
If $\psi:C\rightarrow \inthom{\C}{A}{B}$ is a $C$-partial homomorphism, then we obtain by the cofree-forgetful adjunction a coalgebra homomorphism $C\rightarrow \cofree\big( \inthom{\C}{A}{B}\big)$ which must equalize the two parallel homomorphisms. 
Thus by universal property of equalizers there must exists a unique coalgebra homomorphism $u_\psi:C\rightarrow \ubalg(A,B)$ such that the following diagram commutes:
\[
\begin{tikzcd}
   C \ar[dashed]{d}[swap]{u_\psi}\ar{dr}{\psi} &\\
   \ubalg(A,B) \ar{r}[swap]{\widetilde{\ev}} & \inthom{\C}{A}{B}.
\end{tikzcd}
\]
In other words, we have just shown that the universal property of the equalizer $(\ubalg(A,B), \mathsf{eq})$ is equivalent to the universal measuring $(\ubalg(A,B);\widetilde{\ev})$ which corresponds to the universal measuring $(\ubalg(A,B);\ev)$.
\end{proof}

\subsection{Proof of \texorpdfstring{\cref{prop:sym-mon}}{Proposition \ref{prop:sym-mon}} }\label{subsection appendix: details for symmetric monoidal structure}

We first need to check $(C, \chi_C)\otimes (\I, \eta)\cong (C, \chi_C)$. On objects, we have $C\otimes \I\cong C$ as $\I$ is a monoidal unit. The coalgebraic structures correspond as we have the following commutative diagram by unitality of $(F, \nabla, \eta)$:
\[
\begin{tikzcd}
   C \ar{r}{\cong}\ar{d}[swap]{\chi} & C\otimes \I \ar{dr}{\chi\otimes \eta} \\
   F(C) \ar[bend right=2em, equals]{rrrr} \ar{r}{\cong} & F(C)\otimes \I \ar{r}[swap]{\id\otimes \eta}& F(C)\otimes F(\I) \ar{r}{\nabla} & F(C\otimes \I)\ar{r}{\cong}  & F(C).
\end{tikzcd}
\]
Similarly, we have a natural identification $ (\I, \eta)\otimes (C, \chi_C)\cong (C, \chi_C)$ from the identification $\I\otimes C\cong C$. 
Moreover, we also have natural identifications:
\[
((C, \chi)\otimes (C', \chi')) \otimes (C'', \chi'') \cong (C, \chi)\otimes ((C', \chi') \otimes (C'', \chi'')),
\]
lifted from $(C\otimes C')\otimes C''\cong C\otimes (C'\otimes C'')$.
Lastly, the symmetry $C\otimes C'\cong C'\otimes C$ lifts to $\bcoalg$. Therefore as $\C$ is symmetric monoidal, the above natural identifications assemble to a symmetric monoidal structure on $\bcoalg$ as the pentagon, triangle, hexagon and symmetric identities remain valid in $\bcoalg$.

\subsection{Proof of \texorpdfstring{\cref{thm:enriched}}{Theorem \ref{thm:enriched}}}\label{subsection appendix: algebras enriched over coalgebras}

\begin{proof}
The convolution algebra functor $[-,-]:\bcoalg^\op\times \balg\rightarrow \balg$ presents $\balg$ as a right module over $\bcoalg^\op$. 
Thus $\balg^\op$ is a left module over $\bcoalg$ via \[[-,-]^\op\colon \bcoalg\times \balg^\op \rightarrow \balg^\op.\] 
Moreover, for any $A\in \balg$, the functor $[-,B]^\op\colon\bcoalg\rightarrow \balg^\op$ 
has a right adjoint $\ubalg(-, B):\balg^\op\rightarrow \bcoalg$ as shown in \cref{remark: universal measuring as an adjoint}.
Thus, by \cite[3.1]{measuring}, the category $\balg^\op$ is enriched over $\bcoalg$ with tensor $[-,-]^\op$. 
Therefore $\balg$ is enriched over $\bcoalg$ with power $[-,-]$. 
\end{proof}

\subsection{Details for \texorpdfstring{\cref{theorem: mixed enriched GF over F}}{Theorem \ref{theorem: mixed enriched GF over F}}}\label{subsection appendix: details for mixed enriched GF over F}

If $A$ is a $GF$-algebra, and $C$ is an $F$-coalgebra, then we obtain a $GF$-algebra structure on $[C,A]=\inthom{\C}{C}{A}$ as follows.
Define $GF\big(\inthom{\C}{C}{A}\big) \rightarrow \inthom{\C}{C}{A}$ as the adjunct of:
\[
GF \big( \inthom{\C}{C}{A}\big) \otimes C \rightarrow GF\big( \inthom{\C}{C}{A}\big) \otimes F(C) \rightarrow G\big( F( \inthom{\C}{C}{A}) \otimes F(C) \big) \rightarrow G(F(A)) \rightarrow A.
\]
The first map is induced by the $F$-coalgebra structure $C\rightarrow F(C)$, the last map is induced by the $GF$-structure $GFA\rightarrow A$. 
The penultimate map is applying $G$ to the map  $F( \inthom{\C}{C}{A}) \otimes F(C)\rightarrow F(A)$ induced by the lax symmetric monoidal structure on $F$ and evaluation.
The second map is induced by the $\C$-enrichment of $G$ that produces a map $G(X)\otimes Y\rightarrow G(X\otimes Y)$ for any objects $X$ and $Y$ in $\C$. Indeed, this is the adjunct of:
\[
Y\rightarrow \inthom{\C}{X}{X\otimes Y} \rightarrow \inthom{\C}{G(X)}{G(X\otimes Y)},
\]
where the first map is the coevaluation of the internal hom, and the second map is induced by the $\C$-enrichment.

\subsection{Details for \texorpdfstring{\cref{example: further list of examples}}{Example \ref{example: further list of examples}}}
\label{ssec: details for big example}

\begin{example}
 The following endofunctors on a locally presentable, symmetric monoidal category $\C$ are always accessible and lax symmetric monoidal. 
 \begin{description}
        \item[$(\id)$] The identity endofunctor $\id_\C$ has a trivial lax symmetric monoidal structure, and preserves filtered colimits.
        \item[$(A)$] Given a commutative monoid $(A, \eta, \nabla) \in \C$, the endofunctor that sends each object to $A$ has a lax symmetric monoidal structure. Indeed, $\iota$ and $\nabla$ are exactly the structure required for this functor to be lax symmetric monoidal. This preserves filtered colimits.
        \item[$(GF)$] Suppose that $(F, \eta_F, \nabla_F)$ and $(G, \eta_G, \nabla_G)$ are lax symmetric monoidal endofunctors on $\C$. Then so is the composition $GF$. The map $\eta$ is the composition $\I \xrightarrow{\eta_G} G (\I) \xrightarrow{G \eta_F} GF(\I)$. The map $\nabla$ is the composition $GF(X) \otimes GF(Y) \xrightarrow{\nabla_G} G(F(X) \otimes F(Y)) \xrightarrow{G\nabla_F} GF(X \otimes Y)$. If $F$ and $G$ both preserved filtered colimits, then so does $GF$.
        \item[$(F \otimes G)$] Suppose that $\C$ is closed and that $(F, \eta_F, \nabla_F)$ and $(G, \eta_G, \nabla_G)$ are lax symmetric monoidal endofunctors on $\C$. Then so is $F \otimes G$. The map $\eta$ is the composition $\I \to \I \otimes \I \xrightarrow{\eta_F \otimes \eta_G} (F \otimes G)(\I)$. The map $\nabla$ is the composition $(F \otimes G)(X) \otimes (F \otimes G)(Y) \cong (FX \otimes FY) \otimes (GX \otimes GY) \xrightarrow{\nabla_F \otimes \nabla_G} (F \otimes G)  (X \otimes Y)$. If $F$ and $G$ are accessible, then so is $F \otimes G$ (because $\otimes$ preserves colimits in each variable, as $\C$ is closed). 
        \item[$(F + G)$] Suppose $\C$ is closed and that $(F, \eta_F, \nabla_F)$ is a lax symmetric monoidal endofunctor, and consider a pair $(G, \nabla_G)$ of an endofunctor $G: \C \to \C$ and transformation $GX \otimes GY \to G(X \otimes Y)$ which is associative and symmetric (i.e. a lax symmetric monoidal endofunctor but without $\eta$), and suppose that $G$ is both a left and right module over $F$, meaning that there are natural transformations $\lambda: FX \otimes GY \to G(X \otimes Y)$ and $\rho: GX \otimes FY \to G(X \otimes Y)$ satisfying the diagrams of \cite[Def.~39]{yetter}, where the concept of module over a monoidal functor is defined. Then $F + G$ is a lax symmetric monoidal functor. We set $\eta$ to be the composition $\I \xrightarrow{\eta_F} F (\I) \rightarrow F(\I) + G(\I) \cong (F + G)(\I)$. The map $\nabla$ is given by the following composition.
        \begin{align*}
        \hspace{-1em}
            (F + G)(X) \otimes (F + G)(Y) &\cong (FX \otimes FY) + (FX \otimes GY) + (GX \otimes FY) + (GX \otimes GY) \\ &\xrightarrow{\langle \nabla_F , \lambda, \rho, \nabla_G \rangle} F(X \otimes Y) + G(X \otimes Y) \\&\cong (F+G)(X \otimes Y)
        \end{align*} 
        If $F$ and $G$ preserve filtered colimits, then so does $F + G$.
    \item[$(\id^A)$] Suppose that $\C$ is cartesian (so we denote it $(\C,\term ,\times)$). Consider an exponentiable object $A$ of $\C$, meaning that the functor $\id \times A$ has a right adjoint, denoted $\id^A$. This has a lax symmetric monoidal structure. We let $\term \to \term^A$ and $X^A \times Y^A \to (X \times Y)^A$ be the isomorphisms we get from the fact that $-^A$, as a right adjoint, preserves limits. It is accessible as a right adjoint.
    \item[($W$-types)] Consider the cartesian monoidal structure on $\Set$. Consider a commutative monoid $C \in \Set$, which can be considered as a commutative monoidal discrete category, and a function $f: A \to C$ such that the preimage functor $f^{-1}: C \to \Set$ has an oplax symmetric monoidal structure. Then the polynomial endofunctor (whose initial algebras are often called \emph{W-types}) $\Sigma_C \Pi_f A^*: \C \to \C$ has a lax symmetric monoidal structure. The elements of $\Sigma_C \Pi_f A^*X$ are pairs $(c \in C, g : f^{-1} c \to X)$. The natural transformation $\nabla_{X,Y} : \Sigma_C \Pi_f A^*X \times \Sigma_C \Pi_f A^*Y \to \Sigma_C \Pi_f A^*(X \times Y)$ takes a pair $((c, g : f^{-1} c \to X), (d, h : f^{-1} d \to X))$ to $(cd, (g \times h) \circ \Lambda_{c,d}: f^{-1} (cd) \to X \times Y)$ where $\Lambda_{c,d}: f^{-1} (cd) \to f^{-1} c \times f^{-1}d$ is part of the oplax monoidal structure of $f^{-1}$. The natural transformation $\epsilon: * \to \Sigma_C \Pi_f A^* * \cong C$ is the unit of the monoidal structure on $C$. This polynomial endofunctor is accessible since it is the composition of left and right adjoints.
    \item[(d.e.s.)] Suppose that the monoidal structure on $\C$ is cocartesian, so we denote it by $(\C,0, +)$, and suppose also that $\C$ also has binary products that preserve filtered colimits in each variable. Consider a discrete equational system on $\C$, that is, a pair $(A,M)$ where $A$ is a set and $M$ is an object of $\C^{A \times A}$. The functor category $\C^A$ inherits a cocartesian monoidal structure $(\C^A ,0, +)$ pointwise. Define an endofunctor $M \otimes - : \C^A \to \C^A$ by setting
    \[ (M \otimes X)_a := \sum_{b \in A} M_{b,a} \times X_b. \]
    This is lax symmetric monoidal. The morphisms $\eta: 0 \to M \otimes 0$ and $\nabla: M \otimes (X + Y) \to M \otimes X + M \otimes Y$ are induced by the universal properties of $0$ and $+$. It is also accessible since the coproducts and binary products preserve filtered colimits.
\end{description}
\end{example}

\subsection{Details for \texorpdfstring{\cref{prop:map from initial alg to dual}}{Proposition \ref{prop:map from initial alg to dual}}}
\label{details:prop:map from initial alg to dual}

We consider an object $\lim^C\id_\ubalg$ with the following universal property (if it exists).
\[ \balg(A,\lim^C\id_\ubalg) \cong \lim_{X \in \balg} \ \bcoalg (C, \ubalg(A,X))  \]
Since $\bcoalg (C, \ubalg(A,-))$ preserves limits and $\lim_{X \in \balg} X \cong \N$, we have
\begin{align*}
    \lim_{X \in \balg} \ \bcoalg (C, \ubalg(A,X)) &\cong \bcoalg (C, \ubalg(A,\lim_{X \in \balg} X)) \\
    &\cong \bcoalg (C, \ubalg(A,\N)) \\
    &\cong \balg(A, [C,\N]) \\
    &= \balg(A, C^*)
\end{align*}
Thus, $C^*$ has the universal property of $\lim^C\id_\ubalg$.

Now note that for a $C$-initial algebra $A$, we have $\lim_{X \in \balg} \ \bcoalg (C, \ubalg(A,X)) \cong 1$. Thus, by this universal property of $C^*$, there is a unique map $A \to C^*$.

\end{document}